\documentclass{article} \UseRawInputEncoding
\usepackage{amsmath,amsthm,amssymb,amsfonts}
\usepackage{verbatim}
\usepackage{graphicx}
\usepackage{stmaryrd} 
\usepackage{hyperref}

\theoremstyle{plain}
\newtheorem{thm}{Theorem}
\newtheorem{lem}[thm]{Lemma}
\newtheorem{prop}[thm]{Proposition}

\newtheorem{remark}[thm]{Remark}

\theoremstyle{definition}
\newtheorem{definition}[thm]{Definition}
\newtheorem{exl}[thm]{Example}

\numberwithin{thm}{section}

\newcommand{\adj}{\leftrightarrow}
\newcommand{\adjeq}{\leftrightarroweq}

\DeclareMathOperator{\id}{id}
\def\Z{{\mathbb Z}}
\def\N{{\mathbb N}}

\def \Fix {Fix}

\title{Trimming and 
Building Freezing Sets}

\author{Laurence Boxer
\thanks{Department of Computer and Information Sciences, Niagara University, NY 14109, USA
and  \newline
Department of Computer Science and Engineering, State University of New York at Buffalo \newline
email: boxer@niagara.edu
}
}

\date{ }
\begin{document}
\maketitle{}

\begin{abstract}
We develop new tools for
the construction of fixed
point sets in digital
topology. We define {\em 
excludable} points and show
that these may be excluded
from all freezing sets.
We show that articulation
points are excludable.

We also present results 
concerning points that must 
belong to a freezing set and 
often are easily recognized.
These include
points of degree~1 and some
local extrema. \newline

\newline

\end{abstract}

\date
\maketitle

\section{Introduction}
Freezing sets are part
of the fixed point theory
of digital topology. They
were introduced
in~\cite{BxFPSets2} and
studied subsequently
in~\cite{BxConv, Bx21, BxArbDim, BxConseq, BxColdFreeze, BxLtd,BxCone}.
Given a digital image
$(X,\kappa)$, it is 
desirable to know of a
$\kappa$-freezing set 
for~$X$ that is as small
as possible. A reason
for this: Suppose
we wish to construct
a continuous self-map~$f$
on~$(X,\kappa)$ such that
all members of a
subset~$A$ of~$X$
are fixed by~$f$.
If~$A$ is known to be
a freezing set for~$(X,\kappa)$,
then it can be
concluded that $f=\id_X$, typically
in time depending
on~$\#A$ rather than
on~$\#X$. This could
be a useful savings
of time, since 
often~$\#A \ll \#X$.

In this paper, we study
conditions that can be 
used to exclude or
force membership in a
freezing set. Specifically, let $(X,\kappa)$
be a digital image.
\begin{itemize}
     \item In section~\ref{prelimSec} we give background
     material.
     \item In section~\ref{degree1Include} we show all  
    points of degree~1 in~$(X,\kappa)$
    must be included in all freezing sets of~$(X,\kappa)$.
    \item In section~\ref{excludeSec}, we
    introduce excludable sets. Such a subset
    of~$X$ consists of points that
    may be excluded from any minimal freezing
    set for~$(X,\kappa)$.
    \item In section~\ref{articSec}
    we show all articulation points
    of~$(X,\kappa)$ can be excluded 
    from freezing sets, in the sense that
    if $A$ is a freezing set and~$y$ is 
    an articulation point for~$(X,\kappa)$,
    then $A \setminus \{y\}$ is also
    is a freezing set for~$(X,\kappa)$.
   \item In section~\ref{1-DlocalSec}
   we develop the notion of a
{\em $1$-coordinate local 
extreme point} and show 
that such a point that
satisfies an additional
condition must be a member of every freezing set of~$(X,\kappa)$.
\item In section~\ref{boundsSec}
we develop bounds on the
cardinality of a minimal
freezing set for~$(X,\kappa)$.
\item Section~\ref{furtherRem}
has a brief summary of our
results.
\end{itemize}

\section{Preliminaries}
\label{prelimSec}
We use $\N$ for the set of natural numbers,
$\Z$ for the set of integers, and
$\#X$ for the number of distinct members of $X$.

We typically denote a (binary) digital image
as $(X,\kappa)$, where $X \subset \Z^n$ for some
$n \in \N$ and $\kappa$ represents an adjacency
relation of pairs of points in $X$. Thus,
$(X,\kappa)$ is a graph, in which members of $X$ may be
thought of as black points, and members of $\Z^n \setminus X$
as white points, of a picture of some ``real world" 
object or scene.

\subsection{Adjacencies}
This section is largely
quoted or paraphrased
from~\cite{BxArbDim}.

We use the notations $y \adj_{\kappa} x$, or, when
the adjacency $\kappa$ can be assumed, $y \adj x$, to mean
$x$ and $y$ are $\kappa$-adjacent.
The notations $y \adjeq_{\kappa} x$, or, when
$\kappa$ can be assumed, $y \adjeq x$, mean either
$y=x$ or $y \adj_{\kappa} x$.

For $x \in X$, let
\[ N(X,x,\kappa) =
  \{ \, y \in X \mid x \adj_{\kappa} y \, \}.
\]
The {\em degree} of $x$ in $(X,\kappa)$ 
is~$\#N(X,x,\kappa)$. We are especially
interested in points that have degree~1.

Let $u,n \in \N$, $1 \le u \le n$. 
Let $X \subset \Z^n$. Points
$x = (x_1, \ldots, x_n),~y=(y_1, \ldots, y_n) \in X$ 
are $c_u$-{\em adjacent}
if and only if
\begin{itemize}
    \item $x \neq y$, and
    \item for at most $u$ indices~$i$, 
          $\mid x_i - y_i \mid = 1$, and
    \item for all indices $j$ such that 
          $\mid x_j - y_j \mid \neq 1$, we have
          $x_j = y_j$.
\end{itemize}
The $c_u$ adjacencies are the adjacencies most used
in digital topology, especially $c_1$ and $c_n$.

In low dimensions, it is also common to denote a
$c_u$ adjacency by the number of points that can
have this adjacency with a given point in $\Z^n$. E.g.,
\begin{itemize}
    \item For subsets of $\Z^1$, $c_1$-adjacency is 2-adjacency.
    \item For subsets of $\Z^2$, $c_1$-adjacency is 4-adjacency and
          $c_2$-adjacency is 8-adjacency.
    \item For subsets of $\Z^3$, $c_1$-adjacency is 6-adjacency,
          $c_2$-adjacency is 18-adjacency, and
          $c_3$-adjacency is 26-adjacency.
\end{itemize}

A sequence $P=\{y_i\}_{i=0}^m$ in a digital image $(X,\kappa)$ is
a {\em $\kappa$-path from $a \in X$ to $b \in X$} if
$a=y_0$, $b=y_m$, and $y_i \adjeq_{\kappa} y_{i+1}$ 
for $0 \leq i < m$.

$X$ is {\em $\kappa$-connected}~\cite{Rosenfeld},
or {\em connected} when $\kappa$
is understood, if for every pair of points $a,b \in X$ there
exists a $\kappa$-path in $X$ from $a$ to $b$.

A {\em (digital) $\kappa$-closed curve} is a
path $S=\{s_i\}_{i=0}^{m-1}$ such that $s_0 \adj_{\kappa} s_{m-1}$,
and $i \neq j$ 
implies $s_i \neq s_j$. If also $0 \le i < m$ implies
\[ N(S,x_i,\kappa)=\{x_{(i-1)\mod m},~x_{(i+1)\mod m}\}
\]
then $S$ is a {\em (digital) 
$\kappa$-simple closed curve}. We say
the members of~$S$ 
are {\em circularly
labeled} if they are
indexed as described
above.

Let $X \subset \Z^n$.
The
{\em boundary of} $X$
{\rm \cite{Rosenf79}} is
\[Bd(X) = \{x \in X \, | \mbox{ there exists } y \in \Z^n \setminus X \mbox{ such that } y \adj_{c_1} x\}.
\]

\subsection{Digitally continuous functions}
This section is largely
quoted or paraphrased
from~\cite{BxArbDim}.

Digital continuity is defined
to preserve connectedness, as at
Definition~\ref{continuous} below. By
using adjacency as our standard of ``closeness," we
get Theorem~\ref{continuityPreserveAdj} below.

\begin{definition}
\label{continuous}
{\rm ~\cite{Bx99} (generalizing a definition of~\cite{Rosenfeld})}
Let $(X,\kappa)$ and $(Y,\lambda)$ be digital images.
A function $f: X \rightarrow Y$ is 
{\em $(\kappa,\lambda)$-continuous} if for
every $\kappa$-connected $A \subset X$ we have that
$f(A)$ is a $\lambda$-connected subset of $Y$.
\end{definition}

If either of $X$ or $Y$ is a subset of the 
other, we use the abbreviation
{\em $\kappa$-continuous} for {\em 
$(\kappa,\kappa)$-continuous}.

When the adjacency relations are understood, we will simply say that $f$ is \emph{continuous}. Continuity can be expressed in terms of adjacency of points:

\begin{thm}
{\rm ~\cite{Rosenfeld,Bx99}}
\label{continuityPreserveAdj}
A function $f:X\to Y$ is continuous if and only if $x \adj x'$ in $X$ 
implies $f(x) \adjeq f(x')$.
\end{thm}

See also~\cite{Chen94,Chen04}, where similar notions are referred to as {\em immersions}, {\em gradually varied operators},
and {\em gradually varied mappings}.

A digital {\em isomorphism} (called {\em homeomorphism}
in~\cite{Bx94}) is a $(\kappa,\lambda)$-continuous
surjection $f: X \to Y$ such that $f^{-1}: Y \to X$ is
$(\lambda,\kappa)$-continuous.

The literature uses {\em path} polymorphically: a
$(c_1,\kappa)$-continuous 
function $f: [0,m]_{\Z} \to X$
is a $\kappa$-path if 
$f([0,m]_{\Z})$ is a 
$\kappa$-path from $f(0)$ 
to $f(m)$
as described above.

We use $\id_X$ to denote
the {\em identity function}, $\id_X(x) = x$
for all $x \in X$.

Given a digital image
$(X,\kappa)$, we denote
by $C(X,\kappa)$ the set
of $\kappa$-continuous
functions $f: X \to X$.

Given $f \in C(X,\kappa)$,
a {\em fixed point} of $f$
is a point $x \in X$ such
that $f(x)=x$.
$\Fix(f)$ will denote
the set of fixed points
of~$f$. We say
$f$ is a {\em retraction},
and the set $Y=f(X)$ is a
{\em retract of $X$}, if
$f|_Y = \id_Y$; thus, 
$Y = \Fix(f)$.

\begin{definition}
\label{freezeDef}
{\rm \cite{BxFPSets2}}
Let $(X,\kappa)$ be a
digital image. We say
$A \subset X$ is a 
{\em freezing set for $X$}
if given $g \in C(X,\kappa)$,
$A \subset \Fix(g)$ implies
$g=\id_X$. A freezing set
$A$ is {\em minimal} if
no proper subset of $A$
is a freezing set for
$(X,\kappa)$.
\end{definition}

The following elementary
assertion was noted
in~\cite{BxFPSets2}.

\begin{lem}
\label{swelling}
Let $(X,\kappa)$ be
    a connected digital 
    image for which $A$ is
    a freezing set. If
    $A \subset A' \subset X$, then $A'$ is a
    freezing set for 
    $(X,\kappa)$.
\end{lem}

Let $X \subset \Z^n$,
$x = (x_1, \ldots, x_n) \in Z^n$, where
each $x_i \in \Z$. For
each index~$i$,
the {\em projection map} (onto the $i^{th}$
coordinate) $p_i: X \to \Z$ is given by
$p_i(x) = x_i$.

\subsection{Tools for
determining fixed point sets}
In this section, we give some results, mostly
from earlier papers, that
help us determine fixed point sets for
digitally continuous self maps.
\begin{thm}
\label{freezeInvariant}
{\rm \cite{BxFPSets2}}
Let $A$ be a freezing set for the digital image $(X,\kappa)$ and let
$F: (X,\kappa) \to (Y,\lambda)$ be an isomorphism. Then $F(A)$ is
a freezing set for $(Y,\lambda)$.
\end{thm}

\begin{prop}
{\rm \cite{bs19a}}
\label{uniqueShortest}
Let $(X,\kappa)$ be a digital
image and $f \in C(X,\kappa)$.
Suppose $x,x' \in \Fix(f)$ are
such that there is a unique
shortest $\kappa$-path $P$ in $X$ 
from $x$ to $x'$. Then
$P \subset \Fix(f)$.
\end{prop}

The following lemma may be
understood as saying that
if $q$ and $q'$ are
adjacent with $q$ in a
given direction from $q'$,
and if $f$ pulls $q$ 
further in that direction,
then $f$ also pulls $q'$
in that direction.

\begin{lem}
\label{pullingLem}
{\rm \cite{BxFPSets2}}
Let $(X,c_u)\subset \Z^n$ be a digital image, 
$1 \le u \le n$. Let $q, q' \in X$ be such that
$q \adj_{c_u} q'$.
Let $f \in C(X,c_u)$.
\begin{enumerate}
    \item If $p_i(f(q)) < p_i(q) < p_i(q')$
          then $p_i(f(q')) < p_i(q')$.
    \item If $p_i(f(q)) > p_i(q) > p_i(q')$
          then $p_i(f(q')) > p_i(q')$.
\end{enumerate}
\end{lem}

The following has been
relied on implicitly 
in several previous
papers. It is an
extension of 
Lemma~\ref{pullingLem}
showing that a 
continuous function
can pull a digital arc
that is monotone with
respect to a given
coordinate~$i$
in the direction of 
monotonicity.

\begin{prop}
    \label{pullPath}
    Let $(X,c_u)\subset \Z^n$ be a digital image, 
$1 \le u \le n$. Let $q, q' \in X$.
Let $f \in C(X,c_u)$.
\begin{enumerate}
\item Suppose 
    $\{q_j\}_{j=0}^m$
    is a $c_u$-path 
    in~$X$ such that
    $q_0=q$, $q_{m}=q'$, 
    $p_i(f(q)) < p_i(q)$, and
    for $0 \le j < m$
    we have 
    $p_i(q_j) < p_i(q_{j+1})$.
    Then $p_i(f(q_j)) < p_i(q_j)$ for 
    $0 \le j \le m$.
    \item Suppose 
    $\{q_j\}_{j=0}^m$
    is a $c_u$-path 
    in~$X$ such that
    $q_0=q$, $q_{m}=q'$, 
    $p_i(f(q)) > p_i(q)$, and
    for $0 \le j < m$
    we have 
    $p_i(q_j) > p_i(q_{j-1})$.
    Then $p_i(f(q_j)) > p_i(q_j)$ for 
    $0 \le j \le m$.
\end{enumerate}
\end{prop}

\begin{proof}
We prove the first
assertion; the second
is proven similarly.
We argue by induction.
We know
\[ p_i(f(q_0)) = p_i(f(q)) < p_i(q) =
p_i(q_0).
\]

Suppose we have
$p_i(f(q_k)) < p_i(q_k)$ for some
$k < m$. Then by
Lemma~\ref{pullPath},
$p_i(f(q_{k+1})) < p_i(q_{k+1})$. This
completes our induction.
\end{proof}

\begin{definition}
{\rm \cite{BxConv}}
    Let $\kappa \in \{c_1,c_2\}$. We say a $\kappa$-connected 
set $S=\{x_i\}_{i=1}^n \subset \Z^2$ for $n>1$ is a
{\em (digital) line segment} if the members of $S$ are collinear.
\end{definition}

\begin{remark}
\label{segSlope}
{\rm \cite{BxConv}}
A digital line segment must be vertical, horizontal, or have
slope of $\pm 1$. We say a segment with slope of $\pm 1$ is
{\em slanted}.
\end{remark}

\begin{lem}
\label{c1EndsFixed}
    Let $(X,c_1)\subset \Z^2$ be a connected 
    digital image. 
    Let $L$ be a 
    horizontal or vertical digital 
    line segment of 
    at least 2 
    points contained in~$X$. Let
    $f \in C(X,c_1)$. Suppose
    \begin{equation}
    \label{c1-fixedEndpts}
      \mbox{the endpoints of~$L$ are in~$\Fix(f)$.}  
    \end{equation}
     Then 
    $L \subset \Fix(f)$.
\end{lem}

\begin{proof}
This follows from
Proposition~\ref{uniqueShortest}, since $L$ is
the unique shortest 
$c_1$-path in~$X$ between
the endpoints of~$L$.
\end{proof}

We do not have an
analog of
Lemma~\ref{c1EndsFixed}
for slanted segments,
as shown by the following.

\begin{exl}
    Let
 \[ X = \{(0,0), (0,1), (0,2), (1,0),
 (1,1), (2,0) \}.
 \]
 Let $S = \{(0,2),
 (1,1), (2,0)\}$.
 Let $f: X \to X$ be
 given by
 \[ f(x) = \left \{
 \begin{array}{ll}
     x & \mbox{if }
        x \neq (1,1);\\
     (0,0) & \mbox{if } x = (1,1). 
 \end{array}
 \right .
 \]
 It is easily seen
 that $f \in C(X,c_1)$,
 that $S$ is a slanted segment with
 endpoints $(0,2)$
 and $(2,0)$ in
 $\Fix(f)$, but
 $(1,1) \in S \setminus \Fix(f)$.
\end{exl}

We have the following
analog of
Lemma~\ref{c1EndsFixed}
for the $c_2$ adjacency.

\begin{lem}
\label{c2EndsFixed}
    Let $(X,c_2)\subset \Z^2$ be a connected 
    digital image. Let
    $L$ be a slanted 
    digital line
    segment contained in~$X$. Let
    $f \in C(X,c_2)$. Suppose~(\ref{c1-fixedEndpts}). Then 
    $L \subset \Fix(f)$.
\end{lem}

\begin{proof}
This follows from
Proposition~\ref{uniqueShortest}, since $L$ is
the unique shortest 
$c_2$-path in~$X$ between
the endpoints of~$L$.
\end{proof}

We do not have an analog
of Lemma~\ref{c2EndsFixed}
for horizontal or vertical
digital line segments, as
shown in the following.

\begin{exl}
   Let $X = [0,2]_{\Z}^2
   \subset \Z^2$. 
    The function
    $f: X \to X$ given by
    \[ f(x) = \left \{
    \begin{array}{ll}
    (1,1) & \mbox{if } 
        x \in \{(0,1), (1,0) \}; \\
    x & \mbox{otherwise,}
    \end{array}
    \right .
    \]
    belongs to $C(X,c_2)$. The
    digital line 
    segments
\[ S_1 = \{(0,0), (1,0), (2,0) \},~~~~
   S_2 = \{(0,0), (0,1), (0,2) \}
\]
have endpoints 
in~$\Fix(f)$ and are
respectively horizontal
and vertical, but neither of $S_1, S_2$
is a subset of~   $\Fix(f)$.
\end{exl}

\begin{thm}
\label{bdFreezes}
{\rm \cite{BxFPSets2}}
Let $X \subset \Z^n$ be finite. Then 
for $1 \le u \le n$, $Bd(X)$ is 
a freezing set for $(X,c_u)$.
\end{thm}

\section{Include degree-1 points}
\label{degree1Include}
Sometimes, there are points
of a digital image that are
easily recognized as 
belonging to freezing sets,
as we see in the following.

\begin{thm}
    \label{degree1}
    Let $(X,\kappa)$ be
    a connected digital
    image. Let $A$ be a
    freezing set for
    $(X,\kappa)$. Let
    $x_0 \in X$ be a
    point that has
    degree~1. Then
    $x_0 \in A$.
\end{thm}

\begin{proof}
    Let $x_1$ be the unique member
    of~$X$ such that
    $x_0 \adjeq_{\kappa} x_1$.
    If $x_0 \not \in A$ then the
    function $f: X \to X$ given by
    \[ f(x) = \left \{ \begin{array}{ll}
        x & \mbox{if } x \neq x_0; \\
        x_1 &  \mbox{if } x = x_0,
    \end{array}
    \right .
    \]
    is easily seen to belong to 
    $C(X,\kappa)$. Also, $f|_A = \id_A$,
    and $f \neq \id_X$. The assertion
    follows.
\end{proof}

\section{Excludable
sets}
\label{excludeSec}
We develop the notion of an excludable set and
show how this notion helps us determine
minimal or small freezing sets.

\begin{definition}
\label{excludeDef}
    Let $(X,\kappa)$ be
    a connected digital 
    image. Let $W \subset X$. We say
    $W$ is {\em excludable from freezing sets for $(X,\kappa)$}
    ({\em excludable} for short) if for
    every freezing set
    $A$ of $(X,\kappa)$,
    if $A \setminus W \neq 
    \emptyset$ then
    $A \setminus W$ is
    also a freezing set. If $p \in W$
    then $p$ is an
    {\em excludable point}.
\end{definition}

\begin{remark}
    By Definition~\ref{excludeDef}, if $A$ is
    a minimal freezing
    set and $W$ is 
    excludable for $(X,\kappa)$, then
    $A \cap W = \emptyset$.
\end{remark}

\begin{prop}
\label{excludeEquiv}
    Let $(X,\kappa)$ 
    be a finite connected 
    digital image. 
    Let $W$ be an 
    excludable set 
    for $X$. Then
    every subset
    of~$W$ is excludable.
\end{prop}

\begin{proof}
This follows from
Lemma~\ref{swelling}.
\end{proof}

\begin{prop}
    \label{excludeIsTopological}
    Let $F: (X,\kappa) \to (Y, \lambda)$
    be an isomorphism 
    of digital images.
    If $W$ is
        an excludable
        set for $X$ then $F(W)$ is
        an excludable
        set for $Y$.
\end{prop}

\begin{proof}
Let $A$ be a freezing
set for $X$. Let
$f \in C(Y,\lambda)$
such that 
\begin{equation}
\label{f|}
    f|_{F(A) \setminus F(W)} =
\id_{F(A) \setminus F(W)}
\end{equation}
Then
$g = F^{-1} \circ f \circ F
\in C(X,\kappa)$. 

Let $b \in 
F(A) \setminus F(W)$. Then
$a = F^{-1}(b)
\in A \setminus
W$. We have
\[ g(a) = F^{-1}(f(b)) = F^{-1}(b) = a.
\]
Since every
$a \in A \setminus
W$ satisfies
$a = F^{-1}(b)$
for some
$b \in F(A) \setminus F(W)$, we have
$g|_{A \setminus
W} =
\id_{A \setminus
W}$. Since $W$
is excludable,
$g = \id_X$.
Therefore,
\[ f = F \circ g \circ F^{-1} =
F \circ \id_X \circ 
F^{-1} = \id_Y.
\]
Thus, $F(W)$ is
excludable.
\end{proof}

\section{Articulation points and freezing sets}
\label{articSec}
An {\em articulation point}
or {\em cut point} of a
connected graph 
$(X,\kappa)$ is
a point $x \in X$ such that
$(X \setminus \{x\}, \kappa)$ is not connected
(see Figure~\ref{fig:articulation}).
In this section, we show that
articulation points are often excludable,
by showing that if the set of 
articulation points is removed from
a freezing set, what is
left is often still a freezing
set.

\begin{figure}
    \centering
    \includegraphics{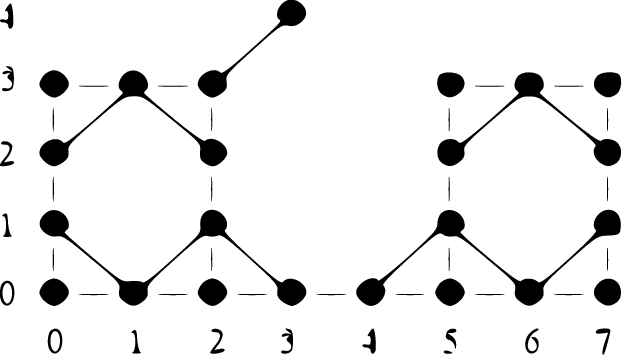}
    \caption{A digital 
    image shown with 
    the $c_2$ adjacency.
    The set of 
    articulation points is
    $\{(2,3), \, (3,0),
    \, (4,0)\}$.
    }
    \label{fig:articulation}
\end{figure}

\begin{lem}
\label{articulRetraction}
    Let $M$ be the set of 
    articulation points for
    the connected digital image
    $(X,\kappa)$. Let
    $K$ be a $\kappa$-component
    of $X \setminus M$.
    Then there is a 
    $\kappa$-retraction of
    $X$ to $X \setminus K$.
\end{lem}

\begin{proof}
Without loss of generality,
$M \neq \emptyset$.

    Since $X$ is connected,
    there exists $x_0 \in X \setminus (K \cup M)$ 
    such that $x_0$ is
    $\kappa$-adjacent to 
    a point of~$M$.
    By choice of $M$, no 
    point of $K$ is 
    adjacent to $x_0$.
    Let $r: X \to X$ be 
    the function
    \[ r(x) = \left \{
    \begin{array}{ll}
        x & \mbox{if } x \in X \setminus K; \\
        x_0 & \mbox{if } x \in K.
    \end{array}
    \right .
    \]
    It is easily seen that
    $r$ is a 
    $\kappa$-retraction of
    $X$ to $X \setminus K$.
\end{proof}

\begin{lem}
\label{articPtFix}
    Let $x_0$ be an
    articulation point for
    the connected digital image
    $(X,\kappa)$.
    Let $K_1$
    and $K_2$ be distinct
    $\kappa$-components 
     of
    $X \setminus \{x_0\}$.
    Let
    $f \in C(X,\kappa)$
    such that for some
    $x_1 \in K_1$ and
    $x_2 \in K_2$,
    $\{x_1,x_2\} \subset \Fix(f)$. Then
    $x_0 \in \Fix(f)$.
\end{lem}

\begin{proof}
Let $P_i$ be a shortest
    $\kappa$-path in $X$
    from $x_i$ to $x_0$,
    $i \in \{1,2\}$. Then
    $P_1 \cup P_2$ is a path
    from $x_1$ to $x_2$. 
    
    By choice of $x_0$ we
    must have 
    $x_0 \in f(P_1)$.
    If $f(x_0) \neq x_0$,  $f(P_1)$ is
    a path from $x_1=f(x_1)$
    to $x_0$ to $f(x_0)$ that
    has length greater than
    that of $P_1$, which is
    impossible. The assertion
    follows.
\end{proof}

\begin{thm}
\label{exclusion}
    Let $W$ be the set
    of articulation
    points for
    the finite connected 
    digital image
    $(X,\kappa)$, with
    $W \neq \emptyset$. If~$W$ is
    a proper subset of~$X$,
    then $W$ is excludable.
\end{thm}

\begin{proof}
    Let $A$ be a
    freezing set for
    $X$. Let $f \in C(X,\kappa)$ such
    that $f|_{A \setminus W} = \id_{A \setminus W}$. 
    Let $x_0 \in W$.
Then there exist distinct 
components $K_1,K_2$ of
$X \setminus \{x_0\}$.
By Lemma~\ref{articulRetraction},
there exists a retraction
$r$ of $X$ to $X \setminus K_1$. It follows that
$A \cap K_1 \neq \emptyset$,
for otherwise
$r|_A = \id_A$
yet $r \neq \id_X$,
contrary to $A$ being
a freezing set. 
Similarly, $A \cap K_2 \neq \emptyset$.

By Lemma~\ref{articPtFix},
$f(x_0) = x_0$. Since $x_0$
was taken as an arbitrary
member of $W$, we have
\[ f|_{A \setminus W} =
    \id_{A \setminus W}
~~~\Rightarrow ~~~
f|_{A \cup W} = 
\id|_{A \cup W}
 ~~~\Rightarrow ~~~ f|_A = \id_A
 ~~~\Rightarrow ~~~ f = \id_X.
    \]
Thus $W$ is excludable.  
\end{proof}

\begin{remark}
Theorem~\ref{exclusion} 
implies that if 
$(X,\kappa)$ is a
wedge of two finite
digital images,
    $(X,\kappa)=
    (X_1,\kappa) \vee
    (X_2,\kappa)$, then
    the ``wedge point" 
    of $X$ is 
    excludable, hence 
    does not belong
    to any minimal 
    freezing set for 
    $(X,\kappa)$.
\end{remark}

\section{1-D local extrema}
\label{1-DlocalSec}
We introduce a kind of local extreme point
and study the relationship of such a point
to a freezing set.

\subsection{Definition and relation to freezing sets}
Let $(X,\kappa)$ be
a digital image, where
$X \subset \Z^n$. Let
$x = (x_1, \ldots, x_n) \in X$, where each
$x_i \in \Z$. We say
$x$ is a 
{\em 1-coordinate
local maximum at 
index~$i$ for $(X,\kappa)$}
if for some index~$i$
and all $y=(y_1, \ldots, y_n) \in N(X,x,\kappa)$,
$x_i > y_i$; and
$x' = (x_1, \ldots, x_{i-1}, x_i -1, x_{i+1}, \ldots, x_n)$ is a
{\em justifying neighbor
at index~$i$} of~$x$.
We say
$x$ is a 
{\em 1-coordinate
local minimum at 
index~$i$ for $(X,\kappa)$}
if for some index~$i$
and all $y=(y_1, \ldots, y_n) \in N(X,x,\kappa)$,
$x_i < y_i$; and
$x' = (x_1, \ldots, x_{i-1}, x_i +1, x_{i+1}, \ldots, x_n)$ is a
{\em justifying neighbor}
of~$x$.
We say
$x$ is a 
{\em 1-coordinate
local extremum for $(X,\kappa)$}
if $x$ is either a
1-coordinate
local maximum or a
1-coordinate
local minimum for $(X,\kappa)$. Such points
are often easily recognized, and, we will
show, often are members
of minimal freezing sets.

Note if $x$ and $x'$ 
are, respectively, a 
1-coordinate local 
extremum and its 
justifying neighbor, then
$x$ and $x'$ differ in
exactly one index. Thus,
$x \adj_{c_1} x'$.

Neither being nor not
being a 1-coordinate 
local extremum is necessarily
preserved by isomorphism, as the 
following shows.

\begin{exl}
\label{C8-c2}
Let 
\[ X = \{(x,y) \in \Z^2 \mid |x| + |y| = 2 \} ~~~~\mbox{(see Figure~\ref{fig:diamondRadius2})}.
\]

\begin{figure}
   \includegraphics{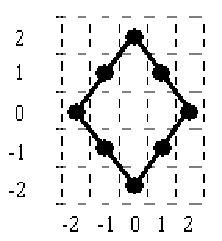}
    \label{fig:diamondRadius2}
    \caption{The digital image $(X,c_2)$
    of Example~\ref{C8-c2}.
   The ``corner" points
    $(2,0),~(0,2),~(-2,0),~(0,-2)$
        are 1-coordinate local
        extrema. None of them has
        a justifying neighbor.}
\end{figure}

Assume the points of $X$ are indexed circularly with
$X = \{x_i\}_{i=0}^7$, with $x_0=(2,0)$, $x_1 = (1,1)$,
etc. Then $(X,c_2)$ is a digital simple closed curve,
for which the ``corner points"
$x_0, x_2, x_4, x_6$ are the 1-coordinate local
extrema. Let $f: X \to X$ be the function
\[ f(x_i) = x_{i+1 \mod 8}.\]
Then $f$ is a $(c_2,c_2)$-isomorphism.
For each index~$i$,
\begin{itemize}
    \item if $x_i$ is a 1-coordinate local
          extremum, then $f(x_i)$ is not 
          a 1-coordinate local extremum; and
     \item if $x_i$ is not a 1-coordinate local
          extremum, then $f(x_i)$ is 
          a 1-coordinate local extremum.
\end{itemize}
\end{exl}

We have the following.

\begin{thm}
\label{localExtremeThm}
    Let $(X,c_u)$ be
a connected digital image, where
$X \subset \Z^n$ and
$1 \le u \le n$. Let
$x_0 = (x_1, \ldots, x_n) \in X$, where each
$x_i \in \Z$. Suppose,
for some index~$i$,
$x_0$ is a 1-coordinate
local extremum for
$(X,c_u)$ with a 
justifying neighbor 
$x' \in X$ 
at index~$i$. Let
$A$ be a freezing set
for $(X,c_u)$. Then
$x_0 \in A$.
\end{thm}

\begin{proof}
There are at most~$u$
indices~$j$, one of
which is $j=i$, 
at which
$|x_j - p_j(x')| = 1$ and
for all other indices~$k$,
$x_k = p_k(x')$.

Therefore, if 
$A \subset X$ and
$x_0 \not \in A$,
consider the function
$f: X \to X$ given by
\[ f(x) = \left \{
\begin{array}{ll}
    x & \mbox{if } x \neq x_0;  \\
    x' & \mbox{if } x = x_0,
\end{array}
\right .
\]
Then for $x \adj_{c_u} y$,
$f(x)$ and $f(y)$
differ in at most
$u-1$ indices. It
follows easily that
$f \in C(X,c_u)$.
Further, $f|_A = \id_A$
and $f \neq \id_X$. Hence
$x_0 \not \in A$ implies
$A$ is not a freezing set.
The assertion
follows.
\end{proof}

\subsection{Example}
We demonstrate how articulation points and
1-coordinate local extrema
can help us determine
small freezing sets in
the following.
\begin{exl}
\label{extremeFreezeFig}
    Let 
    \[ X = 
    \{(0,1), (1,2), (2,1), (3,0), (3,1),
    (3,2),(4,1), (4,2), (4,3), (5,2)\}.
    \]
See Figure~\ref{fig:kite}.

    \begin{figure}
        \centering
        \includegraphics{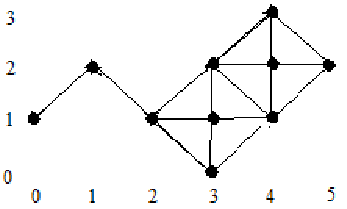}
        \caption{The image $(X,c_2)$ of Example~\ref{extremeFreezeFig}. Articulation points: 
        $(1,2),~(2,1)$.
        1-coordinate local extrema 
        with justifying points:
        $(3,0),~(4,3),~(5,2)$.
        }
        \label{fig:kite}
    \end{figure}

 Let 
   \[ A = \{ (0,1), (3,0), 
   (4,3), (5,2) \}.
\]

   We claim $A$ is a minimal
   freezing set for
   $(X,c_2)$. We show
   this as follows.

   Theorem~\ref{bdFreezes} tells us
   there is a minimal
   freezing set~$B$ for
   $(X,c_2)$ that is a subset of
   $Bd(X)$. Therefore $(3,1)$ and
   $(4,2)$ can be excluded from~$B$. We proceed to
   show $B=A$.

   We must have $(0,1) \in B$, by
   Theorem~\ref{degree1}.

   Since each of $(1,2)$
   and $(2,1)$ is an
   articulation point of
   $(X,c_2)$, by
   Theorem~\ref{exclusion}, 
   they are excluded
   from~$B$. 
   Note also that
   $(1,2)$ is a 1-coordinate local
   maximum in~$X$, but 
   lacks a justifying
   neighbor in~$X$, so
   Theorem~\ref{localExtremeThm} does not apply to
   $(1,2)$.
   
      Each of $(3,0)$, $(4,3)$, and $(5,2)$
   is a 1-coordinate 
   local extreme point
   of~$X$ with justifying
   neighbors $(3,1) \in X$, 
   $(4,2)~\in X$,  and $(4,2)~\in X$,
   respectively, so by
   Theorem~\ref{localExtremeThm}, $\{ (3,0), (4,3),
   (5,2) \} \subset B$.

   Thus $A \subset B$.

Let $f \in C(X,c_2)$ such
that $f|_B = \id_B$.

   Since $(2,1)$ belongs to the
   unique shortest-length $c_2$-path
   between the fixed points $(0,1)$
   and $(3,0)$, 
   Proposition~\ref{uniqueShortest} 
   lets us conclude $(2,1) \in \Fix(f)$.
   Since $(4,1)$ belongs to the
   unique shortest-length $c_2$-path
   between the fixed points $(3,0)$
   and $(5,2)$, 
   Proposition~\ref{uniqueShortest} or
   Lemma~\ref{c2EndsFixed}
   lets
   us conclude $(4,1) \in \Fix(f)$.  
   Similarly, $(3,2)$ belongs to the
   unique shortest-length $c_2$-path
   between the fixed points $(2,1)$
   and $(4,3)$, so
   Theorem~\ref{uniqueShortest} or
   Lemma~\ref{c2EndsFixed}
   lets us conclude
   $(3,2) \in \Fix(f)$.
   Since $f$ is an 
   arbitrary member of
   $C(X,c_2)$, we 
   conclude that
   \[ \{(2,1), (4,1), (3,2)\}
    \subset X \setminus B. 
    \]
   
   Thus $A=B$. Hence by choice of $B$,
$A$ is minimal.
\end{exl}

\section{Bounds on size of freezing set}
\label{boundsSec}
How small, and how big,
can a freezing set be? We provide
bounds on the size of a freezing set in
the following.

\begin{thm}
\label{bounds}
Let $(X,c_u)$ be
a connected finite digital image, where
$X \subset \Z^n$, $n > 1$, and
$1 \le u \le n$. Let $D_1$ be the set of
points that have degree~1 in $(X,c_u)$.
Let $T$ be the set of 1-coordinate
local extrema of~$(X,c_u)$ 
that have justifying
points in~$X$. 
Let $W$ be the set
of articulation points of $(X,c_u)$.
Then there is a minimal
freezing set $A \subset Bd(X)$
for $(X,c_u)$ such that
\[ \#(D_1 \cup T) \le \#A.
\]
If $\emptyset \neq W$ 
and~$W$ is a proper subset of~$X$, then
\[ \#(D_1 \cup T) \le \#A
\le \#Bd(X) - \#W . \]
\end{thm}

\begin{proof}
    By Theorem~\ref{bdFreezes}, there is
    a minimal freezing set
    $A \subset Bd(X)$ for $(X,c_u)$.
    We have $T \subset Bd(X)$,
    $D_1 \subset Bd(X)$, and, since~$n > 1$, $W \subset Bd(X)$.
    Thus the conclusion follows from
    Theorems~\ref{exclusion}, \ref{degree1},
    and~\ref{localExtremeThm}.
\end{proof}

\begin{remark}
    We need $n > 1$ in  
    Theorem~\ref{bounds},
    since if $X = [0,2]_{\Z} \subset \Z$,
    we have that $1$ is an
    articulation point for $(X,c_1)$
    but is not a member of $Bd(X)$.
\end{remark}

\section{Further remarks}
\label{furtherRem}
We have presented the notion
of excludable points in
digital topology, and
have shown that these may be
excluded from all freezing
sets. We have shown that
articulation points are 
excludable. We have shown
that points of degree~1 
and certain local extrema
are points that must be
included in freezing sets.
We have obtained bounds on
the cardinality of a minimal
freezing set for a connected
digital image.


\begin{thebibliography}{99}

\bibitem{Bx94}
L. Boxer, Digitally continuous functions,
{\em Pattern Recognition Letters}
15 (8) (1994), 833-839

\bibitem{Bx99}
L. Boxer, A classical construction for the digital fundamental group, 
{\em Journal of Mathematical Imaging and Vision} 10 (1999), 51-62.

\bibitem{Bx10}
L. Boxer, Continuous maps on digital simple closed curves,
{\em Applied Mathematics} 1 (2010), 377-386.

\bibitem{BxFPSets2}
L. Boxer, Fixed point sets 
in digital topology, 2, {\em Applied General
Topology} 21(1) (2020),
111-133

\bibitem{BxConv}
L. Boxer, 
Convexity and freezing 
sets in digital topology, 
{\em Applied General 
Topology} 22 (1) (2021), 
121 - 137

\bibitem{Bx21}
 L. Boxer, Subsets and 
 freezing sets in the 
 digital plane,
 {\em Hacettepe Journal of 
 Mathematics and 
 Statistics} 50 (4)
 (2021), 991 - 1001

\bibitem{BxArbDim}
L. Boxer, 
 Freezing Sets for 
 arbitrary digital 
 dimension,
 {\em Mathematics} 10 
 (13) (2022), 2291. 

 \bibitem{BxConseq}
  L. Boxer, 
  Some consequences of 
  restrictions on 
  digitally continuous 
  functions, 
  {\em Note di Matematica}
  42 (1) (2022), 47 - 76.

\bibitem{BxColdFreeze}
L. Boxer, Cold and 
freezing sets in the 
digital plane,
{\em Topology Proceedings}
61 (2023), 155 - 182

\bibitem{BxLtd}
L. Boxer, Limiting sets in 
digital topology, 
{\em Note di Matematica}, to appear. 

\bibitem{BxCone}
 L. Boxer, 
 Limiting sets for digital 
 cones and suspensions, 
 submitted. 
\newline
https://arxiv.org/abs/2307.04969 

\bibitem{BxSt19}
L. Boxer and P.C. Staecker,
Remarks on fixed point assertions in digital topology,
{\em Applied General Topology} 20 (1) (2019), 135-153.

\bibitem{bs19a}
 L. Boxer and P.C. Staecker, 
 Fixed point sets in 
 digital topology, 1,
 {\em Applied General
 Topology} 21 (1) (2020),
 87-110



\bibitem{Chen94}
L. Chen, Gradually varied surfaces and its optimal 
uniform approximation, 
{\em SPIE Proceedings} 2182 (1994), 300-307. 

\bibitem{Chen04}
L. Chen, {\em Discrete Surfaces and Manifolds}, Scientific Practical Computing, 
Rockville, MD, 2004

\bibitem{Rosenf79}
A. Rosenfeld,
Digital topology,
{\em The American Mathematical Monthly} 86 (8) (1979), 621 - 630

\bibitem{Rosenfeld}
A. Rosenfeld, `Continuous' functions on digital pictures, 
{\em Pattern Recognition Letters} 4, 1986, 177 - 184

\end{thebibliography}
\end{document}